\documentclass[12pt, leqno]{article}
\usepackage{amsmath,amsthm,amsfonts,amscd,amssymb,comment,eucal,latexsym,mathrsfs}
\usepackage{enumerate,enumitem}

\input xy
\xyoption{all}

\newtheorem{thm}{Theorem}[section]
\newtheorem{prop}[thm]{Proposition}
\newtheorem{lem}[thm]{Lemma}
\newtheorem{cor}[thm]{Corollary}

\theoremstyle{definition}
\newtheorem{defn}[thm]{Definition}

\newtheorem{remark}[thm]{Remark}
\newtheorem{example}[thm]{Example}
\newtheorem{question}[thm]{Question}

\advance \topmargin by -\headheight
\advance \topmargin by -\headsep
\evensidemargin \oddsidemargin
\newcommand\Alt{\operatorname{Alt}}

\def\P{{\cal P}}

\newcommand{\wwr}{\wr\!\wr} 
\newcommand\V{\mathcal{V}}

\def\N{{\mathbb N}}
\def\Z{{\mathbf Z}}
\def\C{{\mathbf C}}

\begin{document}

\title{Constructions of torsion-free countable, amenable, weakly mixing groups}

\author{
\\ Rostislav Grigorchuk 
\footnote{email: grigorch@math.tamu.edu, NSF grant DMS-1207699, ERC AG COMPASP and Swiss National Science Foundation} 
\\ Texas A\&M University
\\ $\ $ 
\\ Rostyslav Kravchenko 
\footnote{email: rkchenko@gmail.com} 
\\ Northwestern University
\\ $\ $ 
\\ Alexander Olshanskii 
\footnote{The third author is supported by
NSF grant DMS-1161294 and Russian RFFR grant  5-01-0582.
email: alexander.olshanskiy@vanderbilt.edu }
\\ Vanderbilt University and Moscow State University}

\date{}

\maketitle

\begin{abstract}
In this note, we construct countable, torsion-free, amenable, weakly mixing groups,
which answer a question of V. Bergelson. Some results related to
verbal subgroups and crystallographic groups are also presented.
\end{abstract}


\section{Introduction}

Weak mixing of a group action on a measure space is a property stronger than ergodicity.
It plays an important role in the modern theory of dynamical systems 
(see for instance \cite{Gl03}, \cite{BG04}, and the references there). 
For actions of cyclic groups, it was introduced by Koopman and von Neumann in \cite{KvN32}. 
Later, von Neumann introduced 
the class of so-called ``minimally almost periodic groups''
(\cite{vN34}, see also \cite{vNW40}), 
which can be characterized by the property that 
every ergodic measure-preserving action of such a group on a finite measure space 
is in fact weakly mixing. 
At present, it is customary to call such groups \emph{weakly mixing groups}, 
or \emph{WM groups} for short. 
At the beginning of the development of the subject,
locally compact groups were involved;
but abstract groups play an important role in recent investigations,
and we restrict the discussion to them in the present article.
The case of amenable groups attracted special attention 
in the paper of Bergelson and Furstenberg \cite{BF}, 
establishing a relation between the WM property and Ramsey theory 
(see also the recent \cite{BCRZ14}). 
\par

For finitely generated groups, 
property WM is the same as having no nontrivial finite quotients
(see \ref{1FROMprop:char} and \ref{2FROMprop:char}
in Proposition \ref{prop:char} below). 
For amenable groups (finitely generated or not),
property WM is equivalent to having no nontrivial finite quotients or abelian quotients 
(\ref{3FROMprop:char} in Proposition \ref{prop:char}).
Thus locally finite simple groups, 
such as the group $\Alt_{\textrm{fin}}(\N)$ of finitary even permutations of $\N$, are WM. 
These groups are torsion groups.
\par

A few years ago V.\ Bergelson, in a private discussion with the first author, 
raised the following question:
 
\begin{question}
\label{q1}
Does there exist an infinite, \emph{torsion-free, amenable,} WM group?
\end{question}

We give a positive answer to this question, 
providing examples satisfying some additional conditions.
\par

This is done in two ways.
First, we follow ideas of 
B. H. Neumann and H.\ Neumann \cite{Ne49, NN59}, 
later developed by P.\ Hall \cite{Ha74} and other researchers.
This leads, see Corollary \ref{coro3.2},
to an example of a countable WM group
which is orderable (and hence torsion-free)
and locally solvable (and hence amenable). 
Additional tools allow us to construct simple groups that answer Question \ref{q1}.
\par

As an alternative, we use groups of type $F'/N'$,
where $F$ is a free group, $N$ a normal subgroup of $F$,
and $N'$ the commutator subgroup of $N$.
Groups of type $F/N'$, and more generally of type $F/\V(N)$ 
where $\V(N)$ is some verbal subgroup of $N$, 
and their subgroups,
were studied intensively in the '60s of the last century 
by many researchers (from \cite{M39} to \cite{Sh65} and much more) 
mostly with the purpose of studying varieties of groups (see \cite{N67}). 
They also play a role in the study of orderable groups,
as can be seen from \cite{KK74} and the literature cited there.
We show that groups of type $F'/N'$ lead to examples of WM groups 
under the condition that $F/N$ is an amenable WM group.
\par

The principal difference between these two constructions is the following.
The first one is an embedding construction,
that is flexible enough to embed groups 
with any combination of the properties in the list below
into groups with the same properties and extra ones.
\begin{equation}
\tag{$\mathcal C$}
\label{mathcalC}
\aligned
&\text{Be torsion-free},
\\
&\text{be locally indicable},
\\
&\text{be amenable,}
\\
&\text{be elementary amenable,}
\\
&\text{be subexponentially amenable,}
\\
&\text{be right orderable,}
\\
&\text{be orderable.}
\endaligned
\end{equation}
In contrast, subgroups of groups given by the second construction are rather special:
they can be regarded as generalizations of torsion-free crystallographic groups 
(see Proposition \ref{cr}).
In particular, every non-free subgroup $H$ of the group $F/N'$ 
has a nontrivial free abelian normal subgroup;
moreover, $H$ must have non-trivial intersection with $N/N'$ (see Proposition \ref{nf}).
\par

In the study of amenable groups, 
an important role is assigned to the splitting of the class $AG$ of amenable groups 
into the disjoint union of the class $EG$ of elementary amenable groups
and the class $AG \smallsetminus EG$ of non-elementary amenable groups. 
A further splitting involves the class $SG$ of subexponentially amenable groups,
so that $AG$ splits into three classes: $EG$, $SG \smallsetminus EG$ and $AG \smallsetminus SG$. 
We provide examples of groups answering Question \ref{q1} 
that belong to these classes. 
A group property stronger than to be torsion-free is the property to be orderable. 
We provide examples with various orderability properties. 
Unfortunately all our examples are infinitely generated, 
and it would be interesting to answer Bergelson's question 
within the class of finitely generated groups.
Such examples would not be right orderable 
since a nontrivial finitely generated right-orderable amenable groups 
can be mapped onto $\mathbb Z$ \cite{Mo06}.
An interesting open question related to the above discussion is:

\begin{question}
\label{q2}
 Does there exist a finitely generated torsion-free, amenable, simple group?
\end{question}

Our note contains also some results concerning verbal subgroups 
(this is related to the second construction of WM groups),
and a construction of crystallographic groups, 
which is also based on the use of groups of type $F'/N'$.

\vskip.2cm

{\bf Acknowledgements}: 
The authors are grateful to L.\ Bowen and P.\ de la Harpe 
for their interest in this work and valuable remarks and suggestions. 
The first and second authors acknowledge the support of Institute Henri Poincar\'e in Paris, 
as a large part of the work on this note was done 
during the trimester program ``Random Walks and Asymptotic Geometry of Groups''.

\section{Preliminaries}
\label{sectionprel}
Since our note lies between group theory and ergodic theory, 
we provide more details and give more definitions than would be required for a paper in one field.
\par

Let $G$ be a group. \emph{Assume first that $G$ is countable} (but see Definition \ref{defWM} below).
Recall that  a measure-preserving measurable action $\alpha$ of $G$ 
on a probability measure space $(X, \mathcal B, \mu)$ is
\begin{enumerate}[noitemsep, label=(\alph*)]
\item[(a)]
\emph{ergodic} if every $G$-invariant measurable subset of $X$ has measure either $0$ or $1$,
\item[(b)]
\emph{weakly mixing} if, for every ergodic
measure-preserving measurable action of $G$ on a probability measure space $(Y, \mathcal C, \nu)$,
the product action of $G$ on $X \times Y$ is again ergodic.
\end{enumerate}
Characterizations in terms of the associated unitary representation $\pi$ of $G$ on
the Hilbert space
$\{ f \in L^2(X, \mathcal B, \mu) \mid \int_X f d\mu = 0 \}$ are standard
(see for example \cite{Sc84}):
\begin{enumerate}[noitemsep, label=(\alph*)]
\item[(a')]
$\alpha$ is ergodic if and only if $\pi$ does not have any non-zero $G$-invariant function,
\item[(b')]
$\alpha$ is weakly mixing if and only if $\pi$ does not have any non-trivial finite dimensional 
subrepresentation.
\end{enumerate}

A countable group $G$ is called \emph{WM}, or \emph{weakly mixing}, 
or \emph{minimally almost periodic},
if one of the following equivalent conditions holds
(i.e.\ if they all hold):
\begin{enumerate}[noitemsep, label=(\roman*)]
\item\label{iFROMPreliminaries}
$G$ has no non-trivial finite-dimensional unitary representations.
\item\label{iiFROMPreliminaries}
$G$ does not admit non-constant almost periodic functions.
\item\label{iiiFROMPreliminaries}
Every ergodic measure preserving action of $G$ on a probability measure space 
is in fact weakly mixing.
\end{enumerate}
Equivalences
\ref{iFROMPreliminaries} $\Leftrightarrow$ \ref{iiFROMPreliminaries}
and
\ref{iiFROMPreliminaries} $\Leftrightarrow$ \ref{iiiFROMPreliminaries}
are proven in \cite{vN34} and \cite{Sc84} respectively.
\par

For example, an infinite cyclic group $\Z$ is not WM. 
Indeed, the action of  $\Z$  on the circle  $\{ z \in \C \mid \vert z \vert = 1 \}$
for which the generator  $1 \in \Z$  acts by a rotation
$z \longmapsto e^{2 \pi i \theta} z$  with  $\theta$ irrational
is ergodic and not weakly mixing.
Note that the group with one element is WM;
other examples of WM groups appear below.
\par

On the one hand, it is necessary to assume that $G$ is countable
for the proofs we know of some of the equivalences stated above;
this is quite explicit in \cite{Sc84}, where groups are assumed to be
locally compact \emph{and second countable}.
On the other hand, for the following definition and for what follows in this article,
the countability assumption is irrelevant.

\begin{defn}
\label{defWM}
A group is weakly mixing, or shortly WM, if it has no non-trivial finite-dimensional unitary representations.
\end{defn}

First we provide an alternative characterization of WM groups 
in the presence of amenability or finite generation.

\begin{prop}
\label{prop:char}
Let $G$ be a group.
\begin{enumerate}[label=(\arabic*)]
\item\label{1FROMprop:char}
If $G$ is WM then $G$ has no non-trivial finite or abelian quotients.
\item\label{2FROMprop:char}
If $G$ is finitely generated, 
then $G$ is WM if and only if it does not have non-trivial finite quotients.
\item\label{3FROMprop:char}
If $G$ does not have non-cyclic free subgroups 
(in particular if $G$ is amenable), 
then $G$ is WM if and only if it does not have non-trivial finite or abelian quotients.
\end{enumerate}
\end{prop}

\begin{proof}
\ref{1FROMprop:char}
We check the contraposition.
If $G$ is a group which has a non-trivial finite or abelian quotient $p : G \twoheadrightarrow Q$,
then $Q$ has a non-trivial finite-dimensional unitary representation $\rho$,
and thus $G$ has the non-trivial finite dimensional unitary pulled back representation $\rho\circ p$.
Hence $G$ is not WM.
\par

\ref{2FROMprop:char}
It suffices again to show the contraposition:
\emph{if $G$ is finitely generated and not WM, then $G$ has a non-trivial finite quotient.}
\par
By hypothesis, there exists a non-trivial unitary representation $\pi : G \to \operatorname{U}(n)$
for some $n \ge 1$.
If $G$ is finitely generated, so is $\pi(G)$. 
Mal'cev proved \cite{Ma40} that all such groups are residually finite. 
In particular, $\pi(G)$ has a non-trivial finite quotient, 
and therefore $G$ also has a non-trivial finite quotient.
\par

\ref{3FROMprop:char}
It suffices to show:
\emph{If $G$ has no non-cyclic free subgroups and is not WM, 
then $G$ has a non-trivial finite quotient or a non-trivial abelian quotient.}
\par

By hypothesis, there exists a non-trivial unitary representation $\pi : G \to \operatorname{U}(n)$
for some $n \ge 1$.
Observe that $\pi(G)$ is non-trivial and has no non-cyclic free subgroups.
By the Tits alternative \cite{Ti72}, $\pi(G)$ is virtually solvable.
We distinguish now two cases:
if $\pi(G)$ has a proper subgroup of finite index, so has $G$ (by pulling back),
and $G$ has a non-trivial finite quotient;
otherwise $\pi(G)$ is solvable and non-trivial, hence $\pi(G)$ has a non-trivial abelian quotient,
and so has $G$.
\end{proof}

The following two corollaries are straightforward consequences of the proposition.

\begin{cor} 
\label{cor:Higman}
Let $H$ be a finitely generated group without finite quotients,
for example the finitely presented 
Higman's group with $4$ generators and $4$ relations constructed in \cite{Hi51}.
For every proper normal subgroup $N$ of $H$, the quotient $H/N$ is a WM group.
\par
In particular, $H$ is a WM group, and, for every maximal normal subgroup $N$ of $H$,
the quotient $H/N$ is a simple WM group.
\end{cor}

Every non-elementary hyperbolic group $G$ has
a non-trivial finitely presented quotient $H$, itself without non-trivial finite quotients \cite{Ol00},
and such a quotient is WM by Corollary \ref{cor:Higman}.
Similarly, there are $2^{\aleph_0}$ non-isomorphic "monsters'' \cite[Theorem 28.7]{Ol89},
and they are WM groups.
A monster is here a non-abelian infinite group in which every proper subgroup is cyclic; 
these groups are $2$-generated and simple.

Recall that a group is called \emph{locally finite} 
if all its finitely generated subgroups are finite;
\emph{locally solvable} groups are defined similarly.
Such groups are amenable, indeed elementary amenable (see the definition below).

\begin{cor}
\label{cor:locally-finite}
Infinite simple locally finite groups are WM.
\end{cor}

It is known that there are uncountably many pairwise non-isomorphic examples
of countable infinite simple locally finite groups.
The simplest example, $\Alt_{\textrm{fin}}(\N)$, has been cited in the introduction.
Other examples are provided by the projective special linear groups $\operatorname{PSL}_n(K)$,
where $n \ge 2$ is an integer and $K$ a locally finite field;
on the one hand, 
there are uncountably many pairwise non-isomorphic locally finite fields
(see \cite{BS89}, in particular Theorem 2.4 and Corollary 2.9);
on the other hand, 
for different $n$ or $K$,  the groups $\operatorname{PSL}_n(K)$ are pairwise non-isomorphic
(see \cite[Satz 2]{SW28}, or $\S$~IV.9 in \cite{Di71}).
For one more class of examples, we refer to
 \cite[Corollary 6.12]{KW73}.
\par

These groups are amenable torsion groups and are not finitely generated;
compare with Question \ref{q2}. 
The following question has been recently answered affirmatively by V. Nekrashevych in \cite{N16}:

\begin{question} 
\label{q3}
Does there exist an infinite, finitely generated, torsion, amenable, simple group?
\end{question}

A remarkable class of infinite finitely generated amenable simple groups 
has recently been discovered by K.\ Juschenko and N.\ Monod \cite{JM13}.
They proved that topological full groups $[[T]]$ 
associated with minimal homeomorphisms $T$ of Cantor sets are amenable,
confirming in such a way a conjecture raised by K.\ Medynets and the first author.
The commutator subgroup of such a group is simple and finitely generated 
when the homeomorphism is a subshift over finite alphabet \cite{Ma06}. Observe however that these groups are neither torsion nor torsion-free.

\vskip.2cm

Recall that a group $G$ is \emph{amenable} if it has invariant mean,
equivalently if it has a left invariant finitely additive probability measure $\mu$
defined on the algebra of all subsets of $G$,
normalized by the condition $\mu(G)=1$. 
The class $AG$ of amenable groups contains
finite groups, abelian groups, and groups of subexponential growth;
it is closed under the following four operations:
(i) taking subgroups, (ii) taking quotients, (iii) extensions, (iv) direct limits 
(the latter operation can be replaced by directed unions).
The class $EG$ of \emph{elementary amenable groups} 
is the smallest class of groups containing finite and abelian groups, 
and closed under the operations (i) to (iv);
it was introduced by M. Day in \cite{Da57}.
The class $SG$ of \emph{subexponentially amenable groups} 
is the smallest class of groups containing finitely generated groups of subexponential growth 
and closed under the operations (i) to (iv); it was introduced in \cite{Gr98}. 
The obvious inclusions $EG \subset SG \subset AG$ are proper \cite{Gr98,BV05}. 
We will say that an amenable group has the 
\emph{type of amenability} $\mathcal{T}_1, \mathcal{T}_2$ or $\mathcal{T}_3$ 
if it is in the class $EG$, $SG \smallsetminus EG$, or $AG \smallsetminus SG$, respectively.
A classical reference about amenable groups is \cite{{Gre69}};
more recent sources of information include the survey \cite{CGH99} and the monograph \cite{CC10}.
\par

Recall that a group is \emph{orderable} 
if it has a linear order (also called a total order)
that is invariant with respect to both left and right multiplication. 
A group is \emph{right (left) orderable}
if it has a linear order invariant with respect to right (left) multiplication. 
To be orderable is a stronger condition than to be right orderable; 
the latter is equivalent to be left orderable, 
and is stronger than to be torsion-free. 
\par

In our first construction (Section \ref{section1stcons} below),
we deal with restricted and unrestricted wreath products of groups. 
It is known that a restricted wreath product of (right) orderable groups is (right) orderable;
an unrestricted wreath product of right orderable groups is right orderable
\cite[Theorem 7.3.2]{MR77},
but an unrestricted wreath product of orderable groups need not be orderable.
Nevertheless, there is a way to set a bi-invariant order 
on some special subgroups of unrestricted wreath products 
(see Part (d) of Lemma \ref{keylemma1stcons}).
The books \cite{KK74} and \cite{MR77}
are good sources of information about orderable groups.

\section{The first construction of torsion-free WM groups}
\label{section1stcons}
Our first construction shows how to embed a group into a simple group 
in such a way that various properties are preserved,
first of all the properties of torsion-freeness and amenability.
This construction uses ideas from \cite{Ne49,NN59} and \cite{Ha74}, 
and some of our statements are simplified versions of statements
that can be found in these articles. 
We present proofs for the reader's convenience.
We begin with the simplest way of obtaining examples 
that answer Question \ref{q1} in the affirmative. 
The corresponding groups are elementary amenable, as they are locally solvable groups.
\par

For a group $G$, 
the \emph{commutator subgroup} is denoted by $G'$ or $[G,G]$.
The group $G$ is \emph{perfect} if $G' = G$.
Recall that the \emph{derived series} $(G^{(s)})_{s \ge 0}$ is defined inductively by
$G^{(0)} = G$ and $G^{(s+1)} = [G^{(s)},G^{(s)}]$ for $s \ge 1$.
A group $G$ is \emph{solvable} if $G^{(s)} = \{1\}$ for $s$ large enough,
and its \emph{solvable length} is then the smallest integer $s$ such that $G^{(s)} = \{1\}$.
The group $G$ is \emph{indicable} if it has an infinite cyclic quotient,
and \emph{locally indicable} if all its finitely generated nontrivial subgroups are indicable.
\par

The next theorem refers to the list (\ref{mathcalC}) of group properties
defined in the Introduction.

\begin{thm}
\label{thm:solv}
Let ($\mathcal D$) be any combination of the group properties of (\ref{mathcalC}).
Every countable group with Property ($\mathcal D$)
 embeds in a countable perfect group with Property ($\mathcal D$).
\end{thm}

\begin{cor} 
\label{coro3.2} 
There is an infinite, countable, orderable, locally solvable and perfect WM group.
\end{cor}

\begin{remark} 
\label{rem1} 
Recall that orderable groups are locally indicable
(Corollary 2, Section 2.2 in \cite{KK74}).
Also right orderable amenable groups are locally indicable \cite{Mo06}. 
In fact, local indicability of the groups involved can be seen directly 
from the construction if we start with indicable group 
and proceed as in the proof of Lemma \ref{keylemma1stcons}.
\end{remark}

The next result is a strengthening of Theorem \ref{thm:solv}:

\begin{thm}
\label{thm:simple}
Let ($\mathcal D$) be as in Theorem \ref{thm:solv}.
Every countable group with Property ($\mathcal D$)
embeds in an infinite countable simple group with Property ($\mathcal D$).
\end{thm}

\begin{cor} 
\label{coro3.5} 
There exists an infinite, countable, orderable, amenable, simple WM group. 
\par

Moreover, such examples exist
in each of the three classes $EG$, $SG \smallsetminus EG$, and $AG \smallsetminus SG$,
as defined near the end of Section \ref{sectionprel}.
\end{cor}

The following
Lemma \ref{keylemma1stcons} is the key argument in proving Theorem \ref{thm:solv}.
It is also the starting point for the construction 
leading to the simple groups mentioned in Theorem \ref{thm:simple}.
We recall first the definitions of wreath products.

Let $A, B$ be two groups.
Their \emph{unrestricted wreath product $A \wwr B$}
is the semi-direct product defined by $A^B \rtimes B$, 
where $A^B$ is the group of maps from $B$ to $A$,
with pointwise multiplication,
and where $\rtimes$ refers to the action of $B$ on $A^B$ by shifts.
Their \emph{restricted wreath product}
is the subgroup $A \wr B := A^{(B)} \rtimes B$ of $A \wwr B$,
where $A^{(B)}$ is the group of maps from $B$ to $A$ with finite support.
\par

Recall that the restricted wreath product of two (right) orderable groups is (right) orderable \cite{Ne49}.
\par

For $a \in A$, we define $\delta_a \in A^{(B)}$ by $\delta_a(1) = a$ and $\delta_a(b) = 1$ for $b \ne 1$.
By the inclusion $a \mapsto \delta_a$,
we identify $A$ with a subgroup of $A \wr B$.
Also, we identify $B$ with a subgroup of $A \wr B$, in the natural way.
A fortiori $A$ and $B$ are also subgroups of $A \wwr B$.

\begin{lem}
\label{keylemma1stcons}
Let $G$ be a group and $H$ a subgroup 
of the unrestricted wreath product $G^{\Z} \rtimes \Z$.
\begin{enumerate}[label=(\alph*)]
\item\label{aFROMkeylemma1stcons}
If $G$ is torsion-free, then $H$ is torsion-free.
\item\label{bFROMkeylemma1stcons}
If $G$ is locally indicable, then $H$ is locally indicable.
\item\label{cFROMkeylemma1stcons}
If $G$ is solvable of derived length $s$, 
then $H$ is solvable of derived length $\le s+1$.
\item\label{dFROMkeylemma1stcons}
If $G$ is amenable, then $H$ is amenable.
\end{enumerate}
Suppose moreover that $G$ is countable.
There exists a countable subgroup $H$ of $G^{\Z} \rtimes \Z$
with the following properties.
\begin{enumerate}[label=(\alph*)]
\addtocounter{enumi}{4}
\item\label{eFROMkeylemma1stcons}
$G$ is a subgroup of $[H,H]$.
\item\label{fFROMkeylemma1stcons}
If $G$ is amenable, 
then $H$ is amenable of the same type of amenability as $G$.
\item\label{gFROMkeylemma1stcons}
If $G$ is (right) ordered,
then $H$ is also (right) orderable with an order extending the order on $G$.
\end{enumerate}
\end{lem}

\begin{proof}
Claims \ref{aFROMkeylemma1stcons} to \ref{dFROMkeylemma1stcons}
are straightforward;  their proofs are left to the reader.

\vskip.2cm

For $g \in G$,
define $f_g \in G^\Z$ by $f_g(n) = g$ if $n \leq 0$ and $f_g(n) = 1$ if $n>0$. 
Let $\sigma$ be the standard generator of the active group $\Z$; 
it acts on $G^\Z$ as the shift to the left, that is $\sigma(f)(n)=f(n+1)$ for all $n\in\Z$.
Observe that $[f_g,\sigma]=f_g\sigma f_g^{-1}\sigma^{-1}=f_g\sigma(f^{-1}_{g})=\delta_g$. 

Let now $H$ be the subgroup of $G^{\Z} \rtimes \Z$ 
generated by $\sigma$ and all $f_g$, $g\in G$.

\vskip.2cm

\ref{eFROMkeylemma1stcons}
By the observation just above, we have $G \leq [H,H]$.

\vskip.2cm

\ref{fFROMkeylemma1stcons}
Let $U$ be a subgroup of $H\cap G^\Z$
generated by a finite set $S$ of elements of the form $f_g^{\sigma^i}$, 
with $g \in G \smallsetminus \{1\}$ and $i \in \Z$. 
Since every $f_g^{\sigma^i} : \Z \to G$ has only two different values, 
the set $\Z$ is a disjoint union of finitely many subsets $Z_k =Z_k(U)$ 
such that every function $\Z \to G$ in $S$, and thus more generally in $U$,
is constant on each $Z_k$.
Therefore $U$ is embeddable into a product of finitely many copies of $G$,
and it follows that $U$ is in the same class, $EG$, $SG$ or $AG$, as $G$ is in.
\par

Since every finitely generated subgroup of $H \cap G^\Z$ is contained
in a subgroup of the kind of $U$ above, 
the countable group $H \cap G^\Z$ is an ascending union 
of finitely generated subgroups of the kind of $U$ above.
It follows that $H \cap G^\Z$ is in the same class, $EG$, $SG$ or $AG$, as $G$ is in.
This holds also for $H$,
because we have an extension $H \cap G^\Z \lhook\joinrel\rightarrow H \twoheadrightarrow q(H)$
in which the right-hand term is a subgroup of $\Z$.
(Here, $q : G^{\Z} \rtimes \Z \twoheadrightarrow \Z$ denotes the canonical projection.)
Since $G \le [H,H]$, we conclude that $H$ is in the same class,
$EG$, $SG \smallsetminus EG$ or $AG \smallsetminus SG$, as $G$ is in.

\vskip.2cm

\ref{gFROMkeylemma1stcons}
Assume that $G$ is right-ordered;
denote by $G_+ = \{ g \in G \mid g > 1 \}$ its cone of positive elements.
Elements of $G^{\Z} \rtimes \Z$ are written $(f,m)$, with $f \in G^\Z$ and $m \in \Z$.
For $f \in G^\Z$ with $f(i) = 1$ for $i$ large enough, 
set $i^{\max}_f = \max \{ i \in \Z \mid f(i) \ne 1 \}$; 
we write $i^{\max}_f = -\infty$ when $f(i) = 1 \in G$ for all $i \in \Z$.
Observe that $i^{\max}_f$ is well-defined for all $f$ in $S$, 
and therefore for all $f \in G^\Z$ with $(f,m) \in H$ for some $m \in \Z$.
Set
$$
H_+ \, = \, 
\{ (f,m) \in H \mid m > 0 \hskip.2cm \text{or} \hskip.2cm
m = 0 \hskip.2cm \text{and} \hskip.2cm f(i^{\max}_f) \in G_+ \} .
$$
It is easy to check that $H_+$ is a subsemigroup of $H$,
that $H_+ \cup H_+^{-1} = H \smallsetminus \{1\}$,
and that $H_+ \cap H_+ = \emptyset$.
It follows that $H_+$ is the cone of positive elements of a total right order on $H$,
defined by $h_1 > h_2$ if $h_1h_2^{-1} \in H_+$. 
This order extends the right order given on $G$. 
\par

Assume moreover that $G$ is ordered,
and more precisely that the order given on $G$ is two-sided, 
equivalently that $G_+$ is invariant by conjugation.
It is again easy to check that $H_+$ is invariant by conjugation,
i.e.\ that $H$ is an orderable group, with the order defined by $H_+$
extending the given order on $G$.
\end{proof}

\begin{proof}[Proof of Theorem \ref{thm:solv}]
Let $G_0$ be a group with Property ($\mathcal D$), 
for example $G_0 = \Z$, with the canonical order.
Define inductively a nested sequence 
$
G_0 < \cdots < G_i < G_{i+1} < \cdots ,
$
where $G_{i+1}$ is obtained from $G_i$ by the same construction as 
$H$ from $G$ in Lemma \ref{keylemma1stcons}.
Define $G$ to be the union $\bigcup_{i \ge 0} G_i$.
Then $G$ is perfect: for any $g \in G$, there exists $j \ge 0$ such that
$g \in G_j < [G_{j+1},G_{j+1}]$.
Since Property ($\mathcal D$) holds for every $G_i$ by Lemma \ref{keylemma1stcons},
it holds also for $G$.
\end{proof}

\begin{proof}[Proof of Corollary \ref{coro3.2}] 
Let $G_0$ be a countable indicable orderable soluble group, e.g.\ $G_0 = \Z$.
Let $(G_i)_{i \ge 0}$ and $G$ be as in the previous proof.
Then $G_i$ is solvable for all $i \ge 1$ 
by Lemma \ref{keylemma1stcons}\ref{cFROMkeylemma1stcons},
so that $G = \bigcup_{i \ge 0} G_i$ is locally solvable.
Moreover, $G$ is orderable, amenable and perfect, by Theorem \ref{thm:solv}.
\par

Since $G$ is perfect, it does not have any nontrivial abelian quotient.
Since $G$ is locally solvable, every finite quotient $K$ of $G$ is solvable;
as moreover $K' = K$, this implies $K = \{1\}$.
Hence $G$ is WM, by Proposition \ref{prop:char}\ref{3FROMprop:char}.
\end{proof}

For the proof of Theorem \ref{thm:simple}, it is convenient to have the following lemma.
\par

\begin{lem}
\label{wreathnormal}
Let $A,B$ be two groups, $G = A \wr B$ their restricted wreath product,
and $N$ a normal subgroup of $G$ containing a non-trivial element $b$ from $B$.
\par

Then $N$ contains $[A,A]$.
\end{lem}

\begin{proof}
Let $x,y\in A$. Then $x$ and $byb^{-1}$ commute, because $b \neq 1$.
Also, $y \equiv byb^{-1} \pmod{N}$. Thus $xy \equiv yx \pmod{N}$, 
or equivalently $[x,y] \in N$.
\end{proof}

\begin{proof}[Proof of Theorem \ref{thm:simple}]
\emph{First step: construction of a group $C$ containing a given group $A$.}
Let $A$ be a group.
For every integer $i \ge 0$, denote by $A_i$ an isomorphic copy
of the group obtained from $A$ as $H$ is obtained from $G$
by the construction of Lemma \ref{keylemma1stcons},
and let $\phi_i : A_i \overset{\approx}{\longrightarrow} A_{i+1}$
be an isomorphism.
Define inductively $W_i$ by $W_0 = A_0$ and $W_i = W_{i-1} \wr A_i$ for $i \ge 1$.
For every $i \ge 0$, identify $W_i$ with a subgroup of $W_{i+1}$
as indicated just before Lemma \ref{wreathnormal},
and define $W = \bigcup_{i=0}^\infty W_i$.
\par

Define inductively monomorphims $\psi_i$ as follows. 
Let $\psi_0: W_0 \hookrightarrow W_1$ extend the isomorphism $\phi_0$, 
by mapping $W_0=A_0$ 
onto the acting group $A_1$ of the wreath product $W_1 = A_0 \wr A_1$.
Assume by induction that 
the monomorphism $\psi_{i-1}: W_{i-1}\hookrightarrow W_i=W_{i-1}\wr A_i$
is already defined for $i \ge 1$. Then the monomorphism 
$\psi_{i}: W_{i}=W_{i-1}\wr A_i\hookrightarrow W_{i+1}=W_{i}\wr A_{i+1}$ 
is given by the pair of monomorphisms $\psi_{i-1}$ and $\phi_i$. 
(Here we identify as above the last $W_i$ with a subgroup of $W_{i+1}$ 
and use the following property of wreath products:
if $X \le Y$ and $Z \le V$ are group pairs, 
then $X$ and Z generate a subgroup in $Y \wr V$ canonically isomorphic to $X \wr Z$.) 
Thus, the series of isomorphisms $\psi_i$ induces an injective endomorphism
$\phi$ on the union $W = \bigcup_{i=0}^{\infty} W_i$ with $\phi (A_i)=A_{i+1}$ for all $i \ge 0$.
\par

Define $C$ to be the HNN extension of $W$
with stable letter $t \in C$ such that $twt^{-1}=\phi(w)$ for every $w \in W$.
\par

Observe that, for any $a \in A_0$, $a \ne 1$, 
the normal closure $N$ of $a$ in $C$ contains $A$.
Indeed, $N$ contains $tat^{-1} \in A_1 \smallsetminus \{1\}$,
so that $N$ contains $A'_0$ by Lemma \ref{wreathnormal},
hence $N$ contains $A$ by Lemma \ref{keylemma1stcons}.
(Recall that $A$ is identified with a subgroup of $A_0$, 
and therefore also with a subgroup of $C$.)

\vskip.2cm

\emph{Second step: construction of a simple group $H$.}
Let us denote by $\theta$ the construction of the first step,
so that a group $A$ is a subgroup of the group $C=\theta(A)$. 
Iteration provides an ascending series $\theta(A) < \theta^2(A) = \theta(\theta(A)) < \cdots$.
Define $H = \bigcup_{i=0}^\infty \theta^i (A)$ to be the union of the groups in this series.
\par

Then $H$ is a simple group. 
Indeed, let $a \in H$, $a \ne 1$;
then $a \in \theta^i(A)$ for some $i$.
The normal closure $N$ of $a$ in $\theta^{i+1}(A)$ 
(and a fortiori in $H$) contains $\theta^i(A)$,
as in the last observation of the previous step.
Similarly, $N$ contains $\theta^j(A)$ for every $j \ge i$.
It follows that $N = H$.

\vskip.2cm

\emph{Third step: if $A$ has some property of (\ref{mathcalC}),
then $H$ has the same property.}
Let us assume that $A$ has some property ($\mathcal P$) 
of the list (\ref{mathcalC}).
For all $i \ge 0$, the group $A_i$ has ($\mathcal P$) by Lemma \ref{keylemma1stcons}.
We claim that so does $W$.
\par

Suppose first that ($\mathcal P$) is local indicability.
Then $W$ has ($\mathcal P$),
because this property is closed under
subgroups, Cartesian products, group extensions and direct unions. 
\par

Suppose now that ($\mathcal P$) is (right) orderability.
This property is stable by restricted (right) products;
see Proposition 4 in Section 1.1 of \cite{KK74},
or proceed as in the proof of Lemma \ref{keylemma1stcons}. 
Consequently, 
if $A_0$ is (right) orderable, say with some (right) order,
then $W_{i+1}$ is (right) orderable, with a order extending that of $W_i$, for all $i \ge 0$.
It follows that $W$ is (right) orderable.
\par

The group $C$ is a semidirect product of 
the group $\overline{W} := \bigcup_{i=0}^{\infty} t^{-i}Wt^i$
and the infinite cyclic group $\langle t\rangle$. 
Because of the properties of the endomorphism $\phi$,
and by induction on $i \ge 1$,
each of the groups $t^{-i}Wt^i$ has a (right) order 
extending the (right) order on its subgroup $t^{-i+1}Wt^{i-1}$. 
Hence the group $\overline{W}$ has a (right) order extending that on $A$. 
Finally the (right) order on $C$ extending this is given by the following rule: 
$t^m w>1$ for all $w \in \overline{W}$ when $m>0$,
and for $w > 1$ in $\overline{W}$ when $m=0$;
we leave it to the reader to check that this indeed defines a positive cone,
and that the resulting order is two-sided if the original order on $A$ is two-sided
(using that the endomorphism $\phi$ preserves the order).
\par

We have shown that, if $A$ is (right) orderable, with some (right) order,
then $\theta(A)$ has a (right) order extending that of $A$.
Similarly, this (right) order extends to $\theta^i(A)$ for all $i \ge 0$, and therfore to $H$.
\par

If ($\mathcal P$) is another property of the list (\ref{mathcalC}),
then it extends from $A_0$ to $H$ by standard arguments,
and the proof is complete.
\end{proof}

\begin{proof}[Proof of Corollary \ref{coro3.5}]
Let first $H$ be the group obtained as in Theorem \ref{thm:simple} from
the group with one element.
Then $H$ is elementary amenable and orderable; 
it is also infinite and simple, 
and therefore it is a WM group by Proposition \ref{prop:char}\ref{3FROMprop:char}.
\par

Let now $\mathcal G$ be any of the $3$-generated $2$-groups of intermediate 
(between polynomial and exponential) growth
constructed by the first author in \cite{Gr84}. 
It is well-known that groups of subexponential growth are amenable,
and the class $EG$ does not contain groups of intermediate growth,
i.e.\ of growth between polynomial and exponential \cite{Ch80};
hence $\mathcal G$ belongs to the class $SG \smallsetminus EG$.
We present it in the form $F/N$, where $F$ is a free group of rank 3. 
It is known that $\mathcal G$ is a residually finite $2$-group,
and therefore the intersection of all the derived subgroups $\mathcal G^{(i)}$ is trivial. 
Hence the group $A=F/N''$ is orderable (see Corollary 2 on Page 109 of \cite{KK74}). 
We have $A \in SG$, since the class $SG$ is closed under extensions, 
and $A\notin EG$, since the homomorphic image $\mathcal G$ is not in $EG$. 
\par

Let $H$ be the group obtained from $A$ as in Theorem \ref{thm:simple}.
Then $H \in SG$, by Theorem \ref{thm:simple},
and $H \notin EG$, since $A$ is a subgroup of $H$. 
Hence $H$ is the required example in $SG \smallsetminus EG$.
\par

Finally, let $\mathcal B$ be the Basilica group that was constructed in \cite{GZ02}. 
It is 2-generated, residually finite-$2$ group, 
amenable (and so, in particular, has exponential growth) \cite{BV05}, but not subexponentially amenable \cite{GZ02}. 
Therefore, if we replace $\mathcal G$ by $\mathcal B$ in the argument of the previous paragraph, 
we obtain the desired example $H \in AG \smallsetminus SG$.
\end{proof}

\begin{remark}  The torsion free group $\tilde{\mathcal{G}}$ of intermediate growth constructed in \cite{Gr85} (it covers $\mathcal{G}$ with abelian kernel) is right orderable, as was shown in in \cite{GM93}.

Note that the Basilica group $\mathcal B$ is {\it right} orderable.
To explain this we use some facts from \cite{GZ02} and the terminology from \cite{BGS03}.
There are two natural embeddings of $\mathcal B'$ in itself, 
given by the geometry of the tree on which $\mathcal B$ acts.
We denote their images by $\mathcal B'_0$ and $\mathcal B'_1$, each isomorphic to $\mathcal B'$.
They are commuting subgroups with trivial intersection in $\mathcal B'$.
Proposition 2 and Lemma~7 from \cite{GZ02} show that 
$\mathcal B$ is weakly regular branch over its commutator subgroup $\mathcal B'$, 
and the relation $\mathcal B' = (\mathcal B_0' \times \mathcal B_1') \rtimes \langle c \rangle$ holds,
where $c$ is the commutator of the two standard generators of $\mathcal B$.
Hence $\mathcal B' / (\mathcal B_0'  \times \mathcal B_1')$ is infinite cyclic,
while $\mathcal B / \mathcal B' \simeq \Z^2$.
It follows that $\mathcal B$ contains a descending sequence $(H_n)_{n \ge 0}$ 
of normal subgroups with trivial intersection, with
$H_0 = \mathcal B$, 
$H_1 = \mathcal B'$, 
and $H_n$ isomorphic to the direct product of $2^{n-1}$ copies of $\mathcal B'$;
moreover, $H_n / H_{n+1} \simeq \Z^{2^{n-1}}$ for $n \ge 1$,
and $H_0 / H_1 \simeq \Z^2$ as already noted.
Since the quotients $H_n/H_{n+1}$ are torsion-free abelian,
$\mathcal B$ is right orderable by a 
result of Zaiceva (Proposition 1, Section 5.4 in \cite{KK74}).
\end{remark}

At present, it is not known whether or not the group $\mathcal B$ is orderable, but the torsion-free group of intermediate growth $\tilde{\mathcal{G}}$ is not as for one of its generators $a^2$ is a central element, but $a$ does not commute with other generators (this was pointed to us by A. Navas).

\begin{question} 
\label{q4}
Does there exist an orderable group of intermediate growth?
\end{question}

\section{The second approach to WM groups}
\label{sectionconstF/N}

The following lemma is well known (see \cite{Hi55}) in case $G$ is a free group. 
The same proof works in the following version.

\begin{lem}
\label{lem:torsion}
Let $G$ be a group such that, for every subgroup $H \leq G$, 
the abelianization $H/H'$ is a torsion-free group. 
Then, for every normal subgroup $N \triangleleft G$, the quotient $G/N'$ is a torsion-free group.
\end{lem}

\begin{proof}
Let $N \triangleleft G$ and $a \in G$; 
set $H = \langle a, N\rangle=\langle a\rangle N$.
It suffices to show that $H/N'$ is torsion-free. 
Consider the exact sequence
\[
1\longrightarrow H'/N' \longrightarrow H/N' \longrightarrow H/H' \longrightarrow 1.
\]
Since $H'\leq N$, we have $H'/N' \leq N/N'$, 
and it follows that $H'/N'$ is torsion-free. 
Hence $H/N'$ is an extension of a torsion-free group by a torsion-free group, 
so that $H/N'$ itself is torsion-free.
\end{proof}

We also need the following result, of independent interest.
To the reader not familiar with the notion of variety of groups,
we suggest, instead of an arbitrary variety, to think of the variety of abelian groups,
replacing in the statement and the proof the notation $\V(G)$, for a verbal subgroup,
by the notation $G'$, for a derived subgroup. Only this special case will be used later.
\par

Recall that, if we have a set of words in a countable group alphabet, 
the corresponding \emph{variety} is the class of all groups 
which have these words $w$ as left-hand sides of identical relations $w=1$ (or laws). 
A variety is \emph{proper} if it is not equal to the class of all groups. 
Let $\V$ be a variety and $G$ a group;
the \emph{verbal subgroup} $\V(G)$ is the subgroup of $G$ generated by 
all values of the words when their letters are replaced by elements of $G$. 
Note that $\V(G)$ is normal, indeed fully characteristic in $G$, 
and that $G/\V(G) \in \V$;
moreover, $G$ is in $\V$ if and only if $\V(G) = \{1\}$.
Let $K$ be a normal subgroup of $G$; then $\V(G/K) = \V(G)K/K$.
If $\V$ is a variety, there is a variety $\V^2$ defined by the equality
$\V^2(G) = \V(\V(G))$ for every group $G$.

\par

We prove the following theorem:

\begin{thm}
\label{thm:var}
Let $F$ be a non-cyclic free group and $N$ a normal subgroup of $F$. 
Let $\V$ be a proper variety of groups.
Then the group $\V(F)/\V(N)$ has a non-trivial quotient in $\V$ if and only if $F/N$ has.
\end{thm}

\begin{proof}
Preliminary observations:
if a group $G$ is in $\V^2$ and non-trivial, 
then $G/\V(G)$ is in $\V$ and non-trivial.
Indeed, $\V( G / \V(G)) = \{1\}$,
and $G/\V(G) = \{1\}$ is impossible
(otherwise $G = \V (G) = \V^2(G) = \{1\}$).

\vskip.2cm

We show first the easy implication:
assuming that there exists a non-trivial quotient
$\pi : \V(F)/\V(N) \twoheadrightarrow Q$, with $Q$ in $\V$,
we have to show that $F/N$ has a non-trivial quotient in $\V$.
For this, the group $F$ can be arbitrary 
(it need not be free and non-cyclic).
\par

We claim that the group $F/\V^2(F)\V(N)$ is not in $\V$.
Indeed, since $\{1\} = \V(Q) \lneqq Q \ne \{1\}$, we have
$$
\V^2(F)\V(N)/\V(N) \, = \,  \V(\V(F)/\V(N)) \, \le \,  \pi^{-1}(\V(Q))
\, \lneqq \, \pi^{-1}(Q) \, = \, \V(F)/\V(N) ,
$$
so that $\V^2(F)\V(N)$ is properly contained in $\V(F)$;
this implies that 
$$
\V(F / \V^2(F)\V(N)) \, = \, \V(F) / \V^2(F) \V(N) \ne \{1\} ,
$$ 
and the claim is proved.
\par

Now we use the claim as follows. 
In the group $F / \V^2(F)\V(N)$, the normal subgroup $\V^2(F)N / \V^2(F)\V(N)$ 
belongs to $\V$, being a homomorphic image of $N / \V(N)$.
It follows from the claim that this normal subgroup is proper, i.e.\  $\V^2(F) N \ne F$.
The nontrivial quotient  $G = F / \V^2(F)N$ belongs to the variety $\V^2$ 
since it is a homomorphic image of $F / \V^2(F)$. 
Therefore, by the preliminary observation, $G$ has a nontrivial homomorphic image in $\V$. 
So has the group $F/N$, as required, since in turn, $G$
is a homomorphic image of $F/N$.

\vskip.2cm

We show now the converse implication, for which we will use
a non-trivial result on non-cyclic free groups.
Assume that $F/N$ has a non-trivial quotient $G\in \V$; 
that is we have a normal subgroup $M\geq N$ with $F/M = G$. 
Then $H := \V(F)/\V(M)$ is in $\V$, because $\V(F) \leq M$. 
The group $H$ is non-trivial by Theorem 43.41 in \cite{N67},
since $M\neq F$ and the variety $\V$ is proper. 
As $\V(N)\le \V(M)$, it follows that $H$ is a quotient of $\V(F)/\V(N)$, 
and this ends the proof.
\end{proof}

\begin{thm}
\label{4.3}
Let $F$ be a non-abelian free group of at most countable rank 
and $N \vartriangleleft F$ a normal subgroup.
\begin{enumerate}[label=(\arabic*)]
\item\label{1FROM4.3}
If $F/N$ is amenable, 
then $F'/N'$ is countable, torsion-free, amenable, of the same type of amenability as $F/N$.
\item\label{2FROM4.3}
 If $F/N$ is a non trivial amenable WM group, 
 then $F'/N'$ is a countable torsion-free, amenable, WM group.
\end{enumerate}
\end{thm}

\begin{proof}
\ref{1FROM4.3} 
Assume that $F/N$ is amenable.
Since $N/N'$ and $F/N$ are amenable, so is the group $F/N'$ of the extension
\begin{equation}
\label{star}
1 \longrightarrow N/ N' \longrightarrow F/N' \longrightarrow F/N \longrightarrow 1,
\end{equation}
and the subgroup $N/N'$ of $F/N'$.
It follows from Lemma \ref{lem:torsion} that the group $F'/N'$ is torsion-free. 
If $F/N$ is elementary amenable, then $F/N'$ and hence $F'/N'$ are elementary amenable.
If $F/N$ belongs to the class $SG \smallsetminus EG$, 
then $F/N'$ also belongs to this class and hence $F'/N'$ 
belongs to $SG \smallsetminus EG$ as $F/N'/F'/N'=F/F'$ is abelian.
Finally, if $F/N$ belongs to the class $AG \smallsetminus SG$
then the same argument shows that $F'/N' \in AG \smallsetminus SG$. 
This proves \ref{1FROM4.3}.

\vskip.2cm

\ref{2FROM4.3} 
We assume that $F/N$ is an amenable WM group.
By \ref{1FROM4.3} and by Proposition \ref{prop:char}\ref{3FROMprop:char},
it suffices to prove that $F'/N'$ does not admit any non-trivial finite or abelian quotient. 
Note that abelian groups form a proper variety; 
by Theorem \ref{thm:var}, indeed by its easy part,
$F'/N'$ cannot have a non-trivial abelian quotient,
otherwise $F/N$ would have a non-trivial abelian quotient, 
in contradiction with Proposition \ref{prop:char}\ref{1FROMprop:char}.
\par

Suppose $F'/N'$ had a non-trivial finite quotient. 
Then $F'/N'$ would have a finite simple non-abelian quotient $H$.
The subgroup $F' \cap N/N'$ of $N/N'$ being normal and abelian, 
$H$ would in fact be a quotient of $F'/(F'\cap N)$, 
which is isomorphic to $F'N/N$, the commutator subgroup of $F/N$. 
Since $F/N$ is WM, we would have $F/N=(F/N)'$, 
and $H$ would be a factor of $F/N$, 
in contradiction with Proposition \ref{prop:char}\ref{1FROMprop:char}.
\par

To prove the second statement, it remains to show that $F'/N'$ is non-trivial. 
Suppose instead that $F' = N'$. Then $N$ contains $F'$ so $F/N$ is abelian. 
Since it is also a nontrivial WM group, we obtain a contradiction 
with Proposition \ref{prop:char}\ref{1FROMprop:char}.
\end{proof}

\begin{remark}
Suppose that $F$ is a non-abelian free group, and $N$ a normal subgroup in $F$.
If $F/N$ has a non-trivial finite quotient, then $F'/N'$ also has a non-trivial finite quotient. 
Indeed we then have that there is a normal subgroup $N\leq R< F$ such that $F/R$ is finite. 
It follows that $F/R'$ is virtually a free abelian group, 
and hence is residually finite. 
Thus $F'/R'\leq F/R'$ is also residually finite. 
Moreover $F'/R'$ is non-trivial by the Auslander-Lyndon result (\cite[Corollary 1.2]{AL55}), 
since $F\neq R$. 
It remains to observe that $F'/R'$ is a homomorphic image of $F'/N'$ since $N'\leq R'$.
\end{remark}

The same conclusion is true if, in the above statement, 
one replaces the variety of abelian groups by any proper variety $\V$ 
(i.e.\ if one replaces the commutator subgroup $N'$ by $\V(N)$);
just use P.\ Neumann's theorem 43.41 from \cite{N67} instead of Auslander-Lyndon's theorem). 
Also, the class of finite groups can be replaced by any star class,
as defined by K. Gruenberg in \cite{Gru57}, 
if this class is closed under homomorphic images.
Gruenberg's star property of an abstract class $\P$ of groups is defined as follows:

A class $\P$ of groups has the \emph{star property} if
\begin{enumerate}[noitemsep,label=(\arabic*)]
\item
$\P$ is closed under taking subgroups and direct products of two groups from $\P$;
\item
if $A$ is a normal subgroup of $B$, if $A$ is a residually $\P$-group and $B/A \in \P$, 
then $B$ is a residually $\P$-group.
\end{enumerate}
Examples of star classes include classes of finite groups, 
of finite $p$-groups, and of solvable groups.
Some results about star classes and residual properties of groups of the form $F/\V(N)$ 
have been obtained by Baumslag, Dunwoody and Andreev-Ol'shanskii 
\cite{Ba63, Du65, AO68}.

By Theorem \ref{4.3}, in order to construct a countable torsion-free, amenable, WM group, 
it suffices to construct a countable amenable, WM group $G$: 
simply present $G$ as $G=F/N$, then $F'/N'$ answers Bergelson's question. 
Here are some examples.

\begin{example}
\label{exAlt}
Let $\Alt_{\textrm{fin}}(\N)$ be the group of 
all finitely-supported even permutations of the natural numbers. 
This group is locally finite and therefore amenable. 
Because it is also simple, Corollary \ref{cor:locally-finite} implies that it is a WM group. 
So if $\Alt_{\textrm{fin}}(\N)=F/N$ then, by Theorem \ref{4.3}, 
$F'/N'$ is a countable torsion-free, amenable, WM group.
\end{example}

\begin{example}
\label{exfull}
Let $T$ be a minimal homeomorphism of the Cantor set $C$, 
i.e.\ a homeomorphism such that the orbit $\{T^ix\,|\, i \in \Z \}$ is dense in $C$ for every $x\in C$. 
Define its full topological group $[[T]]$ as the group of those homeomorphisms $g$ of $C$
such that there exists a closed and open partition $C=\coprod_{s=1}^n C_s$ 
with the property that the restriction of $g$ to any $C_s$ 
coincides with some power $T^{k_s(g)}$ of $T$, 
where $k_s(g)$ is some integer (see \cite{Ma06} or \cite{JM13}). 
Let $[[T]]'$ be the commutator subgroup. 
By \cite{Ma06}, $[[T]]'$ is a countably infinite simple amenable group,
which is finitely generated in case $(T,C)$ is a subshift over a finite alphabet (see \cite{LM95}). 
So, if $F/N = [[T]]'$, then Theorem \ref{4.3} implies that $F'/N'$ 
is a countable torsion-free, amenable, WM group.
\par
The group $F/N = [[T]]'$ is infinite simple amenable, and therefore not elementary amenable
by \cite[Corollary 2.2]{Ch80}. 
Since $(F/N')/(F'/N') \simeq F/F'$ is abelian,
it follows that $F'/N'$ is not elementary amenable.
\end{example}

\begin{prop}
\label{prop:elem}
The group $F'/N'$ from Example \ref{exAlt} is elementary amenable group, 
while the group $F'/N'$ from Example \ref{exfull}
is amenable but not elementary amenable.
\end{prop}

\begin{proof} 
Consider a free group $F$, a normal subgroup $N$, and the quotient $G = F/N$.
\par

If $G$ is elementary amenable, the extension (\ref{star}) of the proof of Theorem \ref{4.3}
shows that $F/N'$ is also elementary amenable;
hence so is its subgroup $F'/N'$.
This occurs for Example \ref{exAlt}, since $\Alt_{\textrm{fin}}(\N)$
is locally finite, and therefore elementary amenable.
\par

If $G = F/N$ is amenable, so is $F/N'$, again by the extension (\ref{star}) of Theorem \ref{4.3}.
If $G = F/N$ is simple non-abelian, it is in particular perfect, so that $F'N/N = F/N$,
and $F'/N'$ factors onto $G$.
If moreover $G$ is not elementary amenable, 
$F'/N'$ has the same property. 
This is the case of $G = [[T]]'$ in Example \ref{exfull}.
\end{proof}

The next statement gives important information about the subgroups of $F/N'$.

\begin{prop}
\label{nf}
Let $H$ be a non-free subgroup of $F/N'$. 
Then the intersection $H\cap N/N'$ is nontrivial 
and therefore $H$ has a nontrivial normal free abelian subgroup.
\end{prop}

\begin{proof}
Let $M=N/N'$ and assume that $H\cap M$ is trivial. 
Then $HM$ is a semidirect product, and we have an exact sequence
\begin{equation}
\label{exact_1}
1 \longrightarrow M \longrightarrow HM \longrightarrow H \longrightarrow 1 .
\end{equation}
Since $HM$ is a subgroup of $F/N'$ of the form $P/N'$ for some $P$, $N\leq P<F$, 
we have an exact sequence
\begin{equation}
1 \longrightarrow M \longrightarrow P/N' \overset{\beta}{\longrightarrow} H
\longrightarrow 1
\end{equation}
with $P$ free and $N\triangleleft P$. 
Let $\gamma : \xymatrix@1{P\ar@{->>}[r] &H}$ be defined as $\gamma=\beta\alpha$, 
where $\alpha:\xymatrix@1{P\ar[r] & P/N'}$ is the canonical projection. Observe that
\[
H \cong (P/N')/(N/N') \cong P/N
\]
and therefore $\operatorname{Ker} \gamma = N$.
\par

We are going to show that, for any $H$-module $A$,
the second cohomology group $H^2(H,A)$ vanishes. 
This will imply that $H$ has cohomological dimension $1$ 
and hence, by Stallings-Swan famous result \cite{S68,S69}, 
that $H$ is a free group, in contradiction witht the hypothesis.
\par

So assume that for some groups $A$ and $G$ with $A$ abelian, 
we have a short exact sequence
 \begin{equation}
 \label{exact_2}
1 \longrightarrow A \overset{i}{\longrightarrow} G
\overset{\pi}{\longrightarrow} H \longrightarrow 1 .
\end{equation}
Then there is a homomorphism $\varphi:\xymatrix@1{P\ar[r] & G}$ making the diagram
\[
\xymatrix{
& & P  \ar[d]_{\varphi} \ar[dr]^{\gamma} & &\\
1\ar[r] & A\ar[r]^{i} & G\ar[r]^{\pi} & H\ar[r] & 1
}
\]
commutative.
Indeed, let $B=\{b_1,b_2,\dots\}$ be a basis of $P$.
Then $\{\gamma(b_j)\}$ generate $H$. 
For each $j$ fix a preimage $g_j\in \pi^{-1}(\gamma(b_j))$ and define $\varphi(b_j)=g_j$.
This defines $\varphi$. Since $\pi\varphi(N)=\gamma(N)=1,$ we have $\varphi(N)\subset i(A)$ and so $\varphi(N')=1$ because $A$ is abelian.
\par

 Therefore the homomorphism $\varphi$ factorizes through $\psi:\xymatrix@1{P/N'\ar[r] & G}$ and there is a homomorphism $\xi:\xymatrix@1{M\ar[r] & A}$
making the diagram
\[
\xymatrix{
1 \ar[r] & M \ar[d]_{\xi} \ar[r] & P/N' \ar[d]_{\psi} \ar[r] & H \ar@{=}[d] \ar[r] & 1
\\
1\ar[r] & A \ar[r] & G \ar[r] & H \ar[r] & 1} 
\]
 commutative.
Now if $\mu: H\to P/N'$ is a splitting homomorphism for the top row, 
i.e.\ $\beta\mu=id$, then $\psi\mu$ splits the bottom row, as required.
\end{proof}

\section{Concluding remarks}
\label{sectionConclRem}
We conclude by including an observation not related to WM groups, 
but related to the use of groups of type $F'/N'$, 
that have appeared in Section \ref{sectionconstF/N}.
\par

Crystallographic group are discrete groups of isometries of 
$n$-dimensional Euclidean spaces which have bounded fundamental domains. 
By a theorem of Bieberbach, they can equivalently be defined strictly in terms of group theory,
and this is the definitions that suits our needs here:

\begin{defn}
\label{cr_def}
A \emph{crystallographic group} is a group $G$ containing a normal subgroup of finite index $N$
which is free abelian of finite rank and is such that the centralizer $C_G(N)$ coincides with $N$.
\end{defn}

Recall that $C_G(N)$ is defined as the group of those $g\in G$ 
which commute with every element of $N$. 
If $N$ is abelian, then clearly $N \leq C_G(N)$;
hence it is the reverse inclusion that matters in the definition above.

\begin{prop} 
\label{cr} 
Let $F$ be a finitely generated free group and $N$ a normal subgroup of finite index.
\par

Then every subgroup of $F/N'$ (for example $F'/N'$) is crystallographic.
\end{prop}

This proposition immediately follows from the following

 \begin{lem}
Let $G$ be a finitely generated torsion-free group, which is virtually abelian. 
Then $G$ is crystallographic.
 \end{lem}
 
\begin{proof}
It follows from our assumptions that 
there exists a maximal normal abelian subgroup $H$ having finite index in $G$. 
Since $G$ is finitely generated and torsion-free, $H$ is free abelian of finite rank.
Suppose $C_G(H)\neq H$. 
The center of $C_G(H)$ has finite index in $C_G(H)$ since it contains $H$. 
Therefore, by a well known theorem of Schur,
$C_G(H)'$ is finite, and indeed trivial since $G$ is torsion-free. 
Thus, $C_G(H)$ is abelian contrary to the choice of $H$.
 \end{proof}

\end{document}